\documentclass[11pt]{amsart}
\usepackage{amssymb,amsmath,amscd,graphicx,fontenc,amsthm,mathrsfs,bbm}
\usepackage[pdftex,bookmarks,colorlinks,breaklinks]{hyperref} 
\hypersetup{linkcolor=blue,citecolor=red}
\usepackage[top=3.5cm,bottom=3.5cm,right=2cm,left=2cm,twoside=false]{geometry}
\newtheorem{theorem}{Theorem}
\newtheorem{lemma}{Lemma}
\newtheorem{definition}{Definition}
\newtheorem{corollary}{Corollary}

\newtheorem{proposition}{Proposition}

\theoremstyle{remark}

\newtheorem{remark}{Remark}
\newcommand{\F}[1]{\mathbb{F}_{#1}}
\newcommand{\FF}{\mathbb{F}}
\newcommand{\GL}{\mathrm{GL}}
\newcommand{\SL}{\mathrm{SL}}
\newcommand{\PGL}{\mathrm{PGL}}
\newcommand{\PSL}{\mathrm{PSL}}
\renewcommand{\P}[1]{\mathbb{P}^{#1}(\mathbb{F}_q)}

\newcommand{\f}[1]{\mathbb{F}_{#1}}
\newcommand{\Suz}{\mathrm{Sz}}

\newcommand{\1}{\mathbf{1}}
\DeclareMathOperator{\Tr}{\mathrm{Tr}}

\title{On the Erd\H{o}s-Ko-Rado property for finite Groups}
\author{Mohammad Bardestani}
\address{Mohammad Bardestani, Department of Mathematics and Statistics, University of Ottawa, 585 King Edward, Ottawa, ON K1N
6N5, Canada.}
\email{mbardest@uottawa.ca}
\author{Keivan Mallahi-Karai}
\address{Keivan Mallahi-Karai, Jacobs University Bremen, Campus Ring I, 28759 Bremen, Germany.}
\email{k.mallahikarai@jacobs-university.de }
\begin{document}
\begin{abstract}
Let a finite group $G$ act transitively on a finite set $X$. A subset $S\subseteq G$ is said to be {\it intersecting} if for any $s_1,s_2\in S$, the element $s_1^{-1}s_2$ has a fixed point. The action 
is said to have the {\it weak Erd\H{o}s-Ko-Rado} property, if the cardinality of any intersecting set is at most $|G|/|X|$. If, moreover, any maximum intersecting set is a coset of a point stabilizer, the action is said to have the {\it strong Erd\H{o}s-Ko-Rado} property. In this paper we will investigate the weak and strong Erd\H{o}s-Ko-Rado property and attempt to classify groups in which all transitive actions have these properties. In particular, we show that a group with
the weak Erd\H{o}s-Ko-Rado property is solvable and that a nilpotent group with the strong Erd\H{o}s-Ko-Rado property is the direct product of a $2$-group and an abelian group of odd order. 
\end{abstract}
\maketitle
\section{Introduction}
A family $\mathcal{A}$ of
$k$-subsets of an $n$-element set is called intersecting if, any two sets in $\mathcal{A}$ have a non-empty intersection. A classical theorem due to Erd\H os, Ko, and Rado~\cite{EKR} asserts that if $n>2k$ the cardinality of an intersecting set is at most ${n-1 \choose k-1}$. Moreover, the only sets for which the equality holds are those consisting of all $k$-subsets containing a fixed element.

Since its appearance, this theorem has been generalized in various ways. Let us recall some of those that are most relevant to the results in this paper. Let $S_n$ denote the symmetric group on the set $X=\{1,2,\dots,n\}$. A subset $S \subseteq S_n$ is called intersecting if, for any two permutations $\sigma, \tau \in S$, there exists $x\in X$ such that $\sigma(x)=\tau(x)$. Motivated by the Erd\H os-Ko-Rado theorem, 
Deza and Frankl~\cite{dezafrankl} proved that if $S \subseteq S_{n}$ is intersecting, then $|S| \leq (n-1)!$. 
In the same paper, they also conjectured that the only intersecting sets of size $(n-1)!$ are the ``canonical'' ones, namely those consisting of all the permutations $ \sigma$ with $ \sigma(i)=j$, for some fixed
$1 \le i,j \le n$. Observe that these are 
exactly the cosets of point stabilizers for this action. 
This conjecture was confirmed by Cameron and Ku~\cite{CamKu}, and independently by Larose and Malvenuto~\cite{Lar}. Later Godsil and Meagher~\cite{God} used the representation theory of symmetric groups to give an alternative proof. In a breakthrough paper,~\cite{Ellis}, Ellis, Friedgut and Pilpel proved an analogous result for $t$-intersecting sets. A set of permutations $S\subseteq S_n$ is called  $t$-intersecting if any two permutations in it agree on at least $t$ points. Their proof also relies on the representation theory of the symmetric groups.

Just as for the standard action of a permutation group on its underlying set, one can define the notion of
intersecting set for arbitrary group actions, and study the structure of maximum intersecting 
sets in this more general context. To be more precise, let $G$ be an arbitrary finite group, acting transitively on a finite set $X$. For $g \in G$ and $ x \in X$, the action of $g$ on $x$ will be denoted by $ g \cdot x$. Let us call a subset $S \subseteq G$ intersecting, if for every $g, h \in S$, there exists a point $x \in X$ such that $g \cdot x= h \cdot x$. This definition immediately shows that a coset of a point stabilizer is  intersecting. In the sequel, we will sometimes, informally, refer to these as the ``canonical'' intersecting sets. 
We say that an action satisfies the weak Erd\H os-Ko-Rado (or shortly, weak EKR) property if the cardinality of any intersecting set is bounded above by the cardinality of the stabilizers. 
If, moreover, the only intersecting sets with the maximum cardinality are the stabilizer cosets (i.e. the canonical ones), then we say that the action has the strong Erd\H os-Ko-Rado (or shortly, strong EKR) property. Finally, we will say that
a group $G$ has the weak (strong, respectively) EKR property, if all transitive actions of $G$ have the weak EKR (strong, respectively) property. 

The aim of this paper is to investigate the prevalence of the weak and strong EKR properties in the category of finite groups and finite group actions. Our first result shows that, perhaps surprisingly, the weak EKR property holds for a large class of finite groups. More precisely, we have the following theorem.

\begin{theorem}\label{wekr}
Let $G$ be a finite group which is either nilpotent or a subgroup of a direct product of groups of square-free order. Then $G$ has the weak EKR property.
\end{theorem}

It is not difficult to construct solvable groups without the weak EKR property. However, we do not yet have a complete characterization of the groups with the weak EKR property. In fact, it is not clear if a neat
classification is at all possible. For now, using the classification of minimal finite simple groups, we can prove the following partial converse to Theorem \ref{wekr}:

\begin{theorem}\label{Solvable}
Let $G$ be a finite group with the weak EKR property. Then $G$ is solvable. 
\end{theorem}

Let us now turn to the strong EKR property. Let $G$ be an arbitrary finite group and
$N$ a normal subgroup of $G$. It is easy to verify (Proposition \ref{normal}) that the action of $G$ on $G/N$ has the strong EKR property. This, in particular, shows that any Hamiltonian group (i.e., a non-abelian group whose all subgroups
are normal) has the strong EKR property. It is a fact that a finite Hamiltonian group is the direct product of the eight-element quaternion group, an elementary abelian $2$-group, and an abelian group of odd order. In particular, such groups are always nilpotent and for any odd prime $p$, their $p$-Sylow subgroup are abelian. Our next theorem offers a partial converse.
\begin{theorem}\label{classification_S}
Let $G$ be a non-abelian nilpotent group with the strong EKR property.
Then $G$ is the direct product of a $2$-group and an abelian group of odd order. 
\end{theorem}

\begin{remark}
It is certainly desirable to have a complete characterization of nilpotent groups with the strong EKR property. Let us remark that such a group does not have to be Hamiltonian. In fact, it is easy to see that the dihedral group of order $8$ has the strong EKR property without being Hamiltonian. 
\end{remark}

Theorems \ref{Solvable} and \ref{classification_S} somewhat suggest that the strong and even the weak EKR property for groups are too restrictive and, thus advocate the study these properties for concrete group actions instead. A rich source of geometric actions is provided by the finite simple groups of Lie type (e.g., $\PSL_{n}$, ${\mathrm{Sp}}_{2n}$, etc. over finite fields) on flag varieties (e.g., projective spaces $\P{n}$ over finite fields). Other finite simple groups are also often realized as permutation groups. 

Indeed, much of the recent work related to the EKR property has focused on such actions. In ~\cite{Me}, Meagher and Spiga considered the action of the projective general
linear group $\PGL_{2}(\F{q})$ on the projective line $\P{1}$ and used character theory to establish the weak and strong EKR property for this action. They proved the weak EKR property of the action of $\PSL_2(\F{q})$ on $\P{1}$, and conjectured that the action has also the strong EKR property. 

More generally, one can consider the action of $\PGL_n(\F{q})$ and $\PSL_n(\F{q})$ on $\P{n-1}$ and study the weak and strong EKR property. As mentioned above, these actions have been considered in ~\cite{Me} in which the authors use character theory to establish the strong EKR property for the action of $\PGL_2(\F{q})$ on $\P{1}$.
Here, using an elementary method, we prove the weak EKR property for a larger class of actions, which proves the easy part of Meagher and Spiga's conjecture.   
  
\begin{theorem}\label{mainPSL} The standard action of $\PGL_n(\F{q})$ on $\P{n-1}$ has the weak EKR property.
Moreover, when $\gcd(n,q-1)=1$, the action of $\PSL_n(\F{q})$ on $\P{n-1}$ has the weak EKR property.
\end{theorem}

It is easy to see that the stabilizer of a hyperplane is also an intersecting set for the action of $\PGL_n(\F{q})$ on $\P{n-1}$, showing that this action does not have the strong EKR property in the above-defined sense. Nevertheless, it may still have a chance to satisfy a weaker variation of it. In ~\cite{Me}, Meagher and Spiga conjecture that the only intersecting sets of the maximum size are the cosets of point and hyperplane
stabilizers. This conjecture has been verified for the case $n=3$ by the same authors in ~\cite{Me2}.
Similarly, using elementary arguments we can prove the following result for a similar action:

\begin{theorem}\label{Unipotent}
Let $G=\GL_2(\f{q})$ or $G=\SL_2(\f{q})$ and $U$ be the subgroup consisting of upper-triangular unipotent matrices.  Then the action of $G$ on $G/U$ has the strong EKR property.
\end{theorem} 

This paper is organized as follows. In Section~\ref{BASE} we give the basic definitions and proof of some general results about the EKR properties. In Section~\ref{EX}, the failure of the weak and strong EKR property for some concrete groups is demonstrated. In Section~\ref{WK}, we prove Theorems~\ref{wekr} and ~\ref{Solvable}.
Section~\ref{ST} is entirely about the proof of Theorem~\ref{classification_S}. Finally, in Section~\ref{ACT}, we will discuss specific actions of the projective general and special linear groups and prove the weak and strong EKR property for these actions.

\section{Notations and preliminary results}\label{BASE}
We start will some notation that will be
used throughout this paper. 
Let $G$ be a finite group. If $\theta$ is an automorphism of $G$, we will write $g^{\theta}$ for $\theta(g)$. 
For the inner automorphism $\iota_h(g)= h^{-1}gh$, we will instead write $g^{h}:=h^{-1}gh$. In particular, for
$g,h,k \in G$, the equation $(g^{h})^{k}= g^{hk}$ holds. For a subset $A \subseteq G$, we will also denote $A^{h}= \{ a^{h}: a \in A \}$. For $A \subseteq G$, we also 
denote $A^{-1}= \{ a^{-1}: a \in A \}$. For $A, B \subseteq G$, denote $AB= \{ ab: a \in A, b \in B \}$. The identity element of groups is denoted by $\1$. The cardinality of a set $A$ is denoted by $|A|$. The finite field with $q$ elements will be denoted by $\F{q}$. Also by a $G$-set $X$, we mean a finite set $X$ with a transitively $G$-action. 

\begin{definition} Let $G$ be a finite group acting transitively on a finite set $X$. We say that the action has the weak EKR property if for any intersecting set $S \subseteq G$, we have $|S| \le |G|/|X|$. We say that the action
has the the strong EKR property if, in addition, the only intersecting sets of size $|G|/|X|$, called maximum intersecting sets, are the cosets of point stabilizers. The group $G$ is said to have the weak EKR (respectively, strong EKR) property if all of its transitive actions have the weak EKR (respectively, strong EKR) property. 
\end{definition}  
\begin{remark}
Note that for any subgroup $H$ of a group $G$, $G$ acts transitively on $G/H$ by translations: $g \cdot (xH)=
(gx)H$\footnote{Every time we refer to ``the action of $G$ on $G/H$'' in the future, we always have this action in
mind.}.
Conversely, any transitive action of a group $G$ on a set $X$ is isomorphic (in the category of $G$-sets) to the action  of $G$ on $G/H$ by translation, for a subgroup $H$ of $G$. The subgroup $H$, which is determined up to conjugacy, is the stabilizer of a point $x \in X$. We will freely use this correspondence throughout the paper.  
\end{remark}

\begin{remark}\label{sgr}
It is easy to see that if $S \subseteq G$ is an intersecting set for the $G$ action on $G/H$, then any translation 
$aS=\{ as: s \in S \}$ is also intersecting. In particular, this action has the strong EKR property if
it has the weak EKR property and the only intersecting sets containing the identity element are the conjugates
of $H$. 
\end{remark}

The proof of the following lemma is straightforward. 

\begin{lemma}\label{inter-1} For a given group $G$ and a subgroup $H \le G$, a subset $S\subseteq G$ is an intersecting set for the action of $G$ on $G/H$ if and only if
$$
S^{-1}S\subseteq \bigcup_{g\in G}H^g.
$$  
\end{lemma} 

In particular, we have

\begin{proposition}\label{normal}
If $H$ is a normal subgroup of $G$, then the action of $G$ on $G/H$ has the strong EKR property. 
In particular, any Hamiltonian group has the strong EKR property.
\end{proposition}

\begin{proof}
Lemma \ref{inter-1} implies that if $S$ is an intersecting set for the action of $G$ on $G/H$, 
then $S^{-1}S \subseteq H$. This shows that a maximum intersecting set $S$ is a coset of the subgroup $H$.  
\end{proof}

The following reformulation of the notion of intersecting sets into the language of graph theory will
also be useful. To an action of a group $G$ on a set $X$, we associate a graph $ \Gamma_{G,X}$ as follows. The vertices of
$\Gamma_{G,X}$ are the elements of $G$. Vertices $g_{1}, g_{2} \in G$ are adjacent if and only if there exists $x \in X$ with $g_{1} \cdot x= g_{2}\cdot x$. Identifying $X$ with $G/H$, this condition is equivalent to $g_{1}^{-1}g_{2} \in 
\bigcup_{g \in G} H^g$. Note that if $g_1$ and $g_2$ are adjacent, then so are $gg_1$ and $gg_2$, for any $g \in G$. This implies that the automorphism group of the graph $\Gamma_{G,X}$ contains $G$ as a subgroup and hence $\Gamma_{G,X}$ is vertex transitive. Recall that a set of vertices of a graph is called a {\it clique} ({\it independent}, respectively), if the induced subgraph on this set is the complete
(empty, respectively) graph. The following well-known lemma will be useful in the sequel. 
For the sake of completeness, we will also provide a proof.

\begin{lemma}\label{cliq}
Let $ \Gamma$ be a vertex transitive graph, $ S \subseteq V( \Gamma)$ a clique, and $ T \subseteq V( \Gamma)$
an independent set. Then, $|S| \cdot |T| \le |V( \Gamma)|.$
\end{lemma} 
\begin{proof}
Let $G$ denote the automorphism group of $ \Gamma$, which, by assumption, acts transitively on $V(\Gamma)$. 
Consider the set $A= \{ (g,s): s \in S, g \in G, g \cdot s \in T \}.$
We have
\begin{equation}\label{double}
 |A|= \sum_{s \in S} \sum_{t \in T} |\{ g \in G: g \cdot s=t \}|
   = \sum_{g \in G} |\{ s \in S: g \cdot s \in T \}|. 
\end{equation}
Since the action is transitive, for any $s,t \in V( \Gamma)$, the set of elements $g \in G$ with 
$g \cdot s=t$ is precisely a coset of the stabilizer of $s$. This implies that 
each summand of the first sum is equal to $|G|/| V( \Gamma)|$ and hence 
\begin{equation}\label{fsum}
|A|= |S||T||G|/| V( \Gamma)|.
\end{equation}
On the other hand, since $S$ is a clique and $T$ an independent set, for every $g \in G$, there is at most
one $s \in S$ with $g \cdot s \in T$. This shows that every summand of the second sum in \eqref{double} is at most $1$ and hence $|A| \le |G|$. Combining this upper bound on $|A|$ with equation \eqref{fsum} establishes the desired inequality.
\end{proof}

Two subgroups $H$  and  $K$  of a group  $G$ are said to be {\it complementary} in $G$, 
if $H\cap K=\{\1\}$ and $KH=G$. The following corollary will be used in the proof of Theorem \ref{wekr}.

\begin{corollary}\label{comp} Let $H$ and $K$ be complementary subgroups of a group $G$.
Then the action of $G$ on $G/H$ has the weak EKR property.
\end{corollary}

\begin{proof}
In order to appeal to Lemma \ref{cliq}, it suffices to show that $K$ is an independent set in the associated  
graph $\Gamma_{G,G/H}$. 
Let $\sigma\in K$ and $xH \in G/H$ be such that $\sigma \cdot xH=xH$. This implies that 
$x^{-1}\sigma x\in x^{-1}Kx\cap H$. Write $x=kh$, with $k\in K$ and $h\in H$. Then,
$$
x^{-1}Kx\cap H=h^{-1}Kh\cap H= h^{-1}(K \cap H) h= \{ \1 \}.
$$ 
Hence $ \sigma=\1$ and the action of $G$ on $G/H$ has the weak EKR property.
\end{proof}

\begin{remark}
The example of Heisenberg group presented at the end of Section \ref{EX} shows that the action in 
Corollary \ref{comp} does not always satisfy the strong EKR property. 
\end{remark}

The following lemma will be used many times in this paper. 
\begin{lemma}\label{inducedEKR}  Let $G$ be a finite group, then the following hold
\begin{enumerate}
\item If $G$ has the weak EKR property, then any subgroup of $G$ also has the weak EKR property. 
\item If the action of $G$ on a set $X$ has the weak EKR property and $N$ is a normal subgroup of $G$ that acts trivially on $X$, then the induced action of $G/N$ on $X$ also has the weak EKR property. 
\item Let $A \le H \le G$ be such that $A$ is a normal subgroup of $G$. If 
the action of $G/A$ on $G/H$ has the weak EKR property, then the action of $G$ on $G/H$ has the weak EKR property.
\end{enumerate}
\end{lemma}
\begin{proof}
Let $H$ be a subgroup of $G$ and consider the action of $H$ on $H/K$ for a subgroup $K \leq H$. 
Let $S \subseteq H$ be an intersecting set for this action. Viewed
as a subset of $G$, $S$ is an intersecting set for the $G$-action on $G/K$. Since $G$ is assumed to
have the weak EKR property, we have $|S| \le |K|$.

For the second part, let $\pi: G \to G/N$  be the quotient map. Let $S$ be 
an intersecting set for the action of $G/N$ on $X$. Then $ \pi^{-1}(S)$ is also an intersecting set 
for the $G$-action on $X$. Hence, $|\pi^{-1}(S)| \le |G|/|X|$. Since $|\pi^{-1}(S)|=|S| \cdot |N|$,
we have $|S| \le |G/N|/|X|$, establishing the weak EKR property. 

For the third part, first observe that each stabilizer of the action of $G/A$ on $G/H$ has size $[H:A].$
Let $S \subseteq G$ be an intersecting set for the action of $G$ on $G/H$, and $\pi: G \to G/A$ be the 
quotient map. Clearly $\pi(S)$ is an intersecting set for the action of $G/A$ on $G/H$. 
By the assumption, we have $|\pi(S)|\leq |H|/|A|$.
Since $S\subseteq \pi^{-1}\left(\pi(S)\right)$, hence $|S|\leq |A| \cdot |\pi(S)|\le |H|$.
\end{proof}

The following lemma shows that the strong EKR property passes to the quotient groups. 

\begin{lemma}\label{s-red}
If $G$ has the weak (strong, respectively) EKR property, then any quotient of $G$ also has the weak
(strong, respectively) EKR property. 
\end{lemma}

\begin{proof}
Let  $G'$ be a quotient of $G$ and $\pi: G \to G'$ be the canonical quotient map. Consider the action of $G'$ on $X'=G'/H'$.
Let $S' \subseteq G'$ be an intersecting set. It is clear that $ \pi^{-1}(S')$ is an intersecting
set for the $G$-action on $X'$. Hence, $|S'| \cdot [G:G']=|\pi^{-1}(S')|\le |G|/|X'|$. This implies that 
$|S'| \le |G'|/|X'|$. So $G'$ has the weak EKR property.
Now, assume that the group $G$ has the strong EKR property. If $S'$ is an intersecting set of size $|G'|/|X'|$, then $\pi^{-1}(S')$ will be an intersecting set
of size $|G|/|X'|$. This shows that $\pi^{-1}(S')= a \pi^{-1}(H'^{g'})$, for some $ a \in G$ and $g'\in G'$. From here, we have $S' = \pi(a)H'^{g'}$. 
\end{proof}

The following well-known lemma 
will be used in Section \ref{WK}. For the convenience of the reader, we will provide a proof. 

\begin{lemma}\label{sum2}
If $F$ is a finite field, then any element $a \in F$ can be expressed as the sum of two 
squares. 
\end{lemma}

\begin{proof}
When the characteristic of $F$ is $2$, the map $s(x)=x^2$ is the Frobenius automorphism which 
is surjective. Assume that $|F|$ is an odd number. It is easy to see that 
the cardinality of the set of squares $S$
is $ \frac{q+1}{2}$. For $x \in F$, set, $x-S= \{ x-s: s\in S \}$. Since $ |S | + |x-S|= q+1>q$, we have
$S \cap (x-S) \neq \emptyset$. Assuming $s_1=x-s_2$, we have $x=s_1+s_2$, which establishes the claim. 
\end{proof}
\section{Rational Canonical forms and the Erd\H{o}s-Ko-Rado Properties}\label{EX}
In this section we will make use of various matrix canonical forms to demonstrate the absence of the weak and strong EKR property for some specific group actions. Although the result of this section will be eventually subsumed by Theorems \ref{classification_S} and \ref{Solvable} which will be proved in 
subsequent sections, the proofs given here contain some of the ideas behind the proof of the general theorems, while 
being at the same time more elementary. The following lemma will be used repeatedly throughout this section. 

\begin{lemma}\label{wekrl}
If a finite group $G$ has subgroups $U$ and $V$ such that $|V|>|U|$ and $ V \subseteq \bigcup_{g \in G } U^g$.
Then $G$ does not have the weak EKR property.
\end{lemma}

\begin{proof}
The second assumption coupled with Lemma~\ref{inter-1} imply that $V$ is an intersecting set in the action of $G$ on 
$G/U$. The first assumption will show that the action does not have the weak EKR. 
\end{proof}

We will first use this lemma to show the absence of the weak EKR property for some non-solvable groups. 

\begin{theorem}
The following groups do not have the weak EKR property:
\begin{enumerate}
\item The group $\PGL_2(\f{q})$ for prime powers $q=p^l$, where $l \ge 2$.
\item The group $\PGL_n(\f{q})$ for $n \ge 3$ and prime power $q$.
\end{enumerate} 
\end{theorem}

\begin{proof}
The proofs are based on the rational canonical forms. 
For part (1), let $G=\PGL_2(\f{q})$ and consider the subgroups defined by
\[ U= \left\{ \begin{pmatrix}
1  & x    \\
 0 & 1 \\
\end{pmatrix}: x \in \FF_p \right\}, \qquad
 V= \left\{ \begin{pmatrix}
1  & x    \\
 0 & 1 \\
\end{pmatrix}: x \in \FF_q \right\}. \]
Let $g \in V \setminus \{ I_{2} \}$. Since $g$ and $u=\begin{pmatrix}
1  & 1    \\
 0 & 1 \\
\end{pmatrix} \in U$ have the same rational canonical form, they are conjugate. Since $u \in U$ and $ l \ge 2$, the pair $(U,V)$ satisfies the condition of Lemma \ref{wekrl}. 

We will write the proof of Part (2) for $n=3$. The proof for general $n$ is analogous. Consider the subgroups of $G= \PGL_3(\f{q})$ defined by
\begin{equation}\label{PGL_3}
U= \left\{ \begin{pmatrix}
1  & x  & 0  \\
 0 & 1  & 0\\
0 & 0 & 1
\end{pmatrix}: x \in \f{q} \right\}, \qquad
V= \left\{ \begin{pmatrix}
1  & x  & z  \\
 0 & 1  & 0\\
0 & 0 & 1
\end{pmatrix}: x,z \in \f{q} \right\}.
\end{equation}
Note that there is only one possible Jordan normal form for a unipotent matrix $g\in V \setminus \{ I_3 \}$, that is 
$$
\begin{pmatrix}
1 & 1 & 0\\
0 & 1 & 0\\
0 & 0 & 1
\end{pmatrix}\in U.
$$
This shows that $V$ is an intersecting set, with 
$|V|>|U|$, which again shows that $G$ does not have weak ERK.
\end{proof}

\begin{remark}\label{remark-PSL3}
A variation of this argument will be needed in the proof of Lemma \ref{EKR-simple}. Let $p=3$ and $n=3$. 
If $X,Y \in \GL_3(\F{3})$ and $X=P^{-1}YP$, replacing $P$ by $-P$, if necessary, we can always assume that $\det P=1$. This shows that with $U$ and $V$ as defined above, we have
$$V \subseteq \bigcup_{g \in \PSL_3(\F{3})}
U^g.$$ 
Hence $\PSL_3(\mathbb{F}_3)$ does not have the weak EKR property. 
\end{remark}

Let us now turn to the case of nilpotent groups and study the lack of the strong EKR property in a special case. Let $p>2$ be a prime number and consider the Heisenberg group defined by
\[ G_p= \left\{ \eta(x,y,z):=\begin{pmatrix}
1  & x  & z  \\
 0 & 1  & y\\
0 & 0 & 1
\end{pmatrix}: x,y, z \in \mathbb{F}_p \right\}. \]
It is obvious that $|G_p|=p^3$ and its center is given by 
$$Z:=Z(G_p)=\left\{\eta(0,0,z): z\in\mathbb{F}_p\right\}.$$
Hence $|G_p/Z(G_p)|=p^2$, implying that $G_p$ is nilpotent of class $2$, and therefore it has the weak EKR property by Theorem~\ref{wekr}. 
We will now construct a concrete action of $G_p$ with the weak but without the strong EKR property. 
Set
$$
H:=\left\{\eta(x,0,0): x\in\mathbb{F}_p\right\},\qquad L:=\left\{\eta(x,0,z): x,z\in\mathbb{F}_p\right\}.
$$ 
A simple computation shows that 
\begin{equation}\label{heisenberg}
\eta(x,0,0)^{\eta(a,b,c)}=\eta(x,0,bx).
\end{equation}
One can easily deduce from \eqref{heisenberg} that
\begin{equation}\label{L/Z}
\bigcup_{g\in G_p} H^g=(L\setminus Z) \cup \{ I_3 \}.
\end{equation}
Let $S=\{g_i:=\eta(a_i,b_i,c_i): 1\leq i\leq k\}$ be an intersecting set
for the action of $G_p$ on $G_p/H$. A simple computation shows that for $1 \le i, j \le k$, we have
\begin{equation}\label{eq1}
g_i^{-1}g_j=\eta(-a_i,-b_i,a_ib_i-c_i)\eta(a_j,b_j,c_j)=\eta(a_j-a_i,b_j-b_i,c_j-c_i+a_ib_i-a_ib_j)
\end{equation}
Since $S$ is intersecting, for $ 1 \le i \neq j \le k$, 
we have $g_i^{-1}g_{j}\in L\setminus Z$, hence $b_i=b_j$.
Equation~\eqref{eq1} now simplifies to
$$
g_i^{-1}g_j=\eta(a_j-a_i,0,c_j-c_i).
$$
Let us first show that this action has the weak EKR property, without applying Theorem~\ref{wekr}. Since there are exactly $p$ possible values for $a_i$, if $k>p$ the pigeonhole principle implies that there
exists $l\neq t$ for which $a_l=a_t$. This shows that $g_l^{-1}g_t=\eta(0,0,c_t-c_l)\in L\setminus Z$ which,
in turn, implies that $c_t=c_l$. This is a contradiction, since $g_t\neq g_l$.

On the other hand, this action does not have the strong EKR property. 
To show this, set
$$
S:=\left\{\eta(x,0,x^2): x\in\mathbb{F}_p\right\}. $$
By~\eqref{L/Z} and~\eqref{eq1}, $S$ is an intersecting set containing the identity matrix which is not a subgroup of $G_p$. Remark \ref{sgr} will imply that the action of $G_{p}$ on $G_{p}/H$ does not have the 
strong EKR property.

\section{The weak Erd\H{o}s-Ko-Rado Property for Finite Groups }\label{WK}
This section consists of two complementary parts. First we will give the proof of Theorem~\ref{Solvable}, showing that 
a non-solvable group cannot have the weak EKR property. Then, we will show that nilpotent groups and some solvable groups
have the weak EKR property.

\subsection{Absence of the Weak EKR Property for non-Solvable Groups }

 The main ingredient of the proof is the following 
lemma.

\begin{lemma}\label{EKR-simple}
The following groups do not have the weak EKR property:
\begin{enumerate}
\item $\PSL_2(\F{q})$, $q\geq 3$ a prime power,
\item $\PSL_3(\F{3})$,
\item The Suzuki group, $\Suz_{2^{p}}$, $p$ an odd prime,
\item The alternating group $A_n$ for $n \ge 5$. 
\end{enumerate}
\end{lemma}

A non-abelian simple group $G$ is called a {\it minimal finite simple group}
if every proper subgroup of $G$ is solvable. Minimal finite
simple groups were classified by John Thompson 
in a series of groundbreaking papers which also laid the 
foundation for the subsequent study of finite simple groups. It turns out that
the proof of Lemma \ref{EKR-simple} does not require
the complete classification of finite simple groups. The following result
of Thompson would suffice.
\begin{theorem}[\cite{Thompson}, Corollary 1]\label{Thompson}
Every minimal finite simple group is isomorphic to one of the 
following finite simple groups:
\begin{enumerate}
\item $\PSL_{2}(\f{2^{p}})$, $p$ any prime.
\item $\PSL_{2}(\f{3^{p}})$, $p$ any odd prime. 
\item $\PSL_{2}(\f{p})$, $p>3$ any prime such that 
$p^{2}+1 \equiv 0 \pmod{5}$.
\item $\PSL_{3}(\f{3})$.
\item $\Suz_{2^{p}}$, $p$ any odd prime.  
\end{enumerate}
\end{theorem}

Note that Lemma \ref{EKR-simple} covers all the finite simple groups appearing in 
Theorem \ref{Thompson}.

\begin{proof}[Proof of Lemma \ref{EKR-simple}]
We first consider the case of $\PSL_2(\F{q})$. The proof goes as follows: we will construct a subgroup $H\leq \SL_2(\F{q})$ containing the center of $\SL_2(\F{q})$ such that the action of $\SL_2(\F{q})$ on $\SL_2(\F{q})/H$ doesn't have the weak EKR property. Then from part (3) of Lemma~\ref{inducedEKR} we conclude that $\PSL_2(\F{q})$ doesn't have the weak EKR property. The proof for $\SL_2(\F{q})$ breaks into three sub-cases:\\

\noindent{\bf Case 1}: {\it $-1$ is not an square in $\F{q}$.}

In this case, let $H$ be the $4$-element subgroup of $\SL_2(\F{q})$ generated by the matrix 
$$J:=\begin{pmatrix}
0 & -1\\
1 & 0
\end{pmatrix}.$$
We remark that if $-1$ is not a square in $\F{q}$, then any matrix $X \in \GL_2(\F{q})$ with trace zero and determinant $1$ has the minimal polynomial $x^2+1$ and hence it is conjugate to $J$, that is, $X= P^{-1}JP$, for $P \in \GL_2(\F{q})$. 
We claim that $P$ can be chosen so that $\det P=1$. To prove this, first observe that,
by Lemma~\ref{sum2}, the centralizer of $J$, which is given by
$$
C_{\GL_2({\F{q}})}(J)=\left\{\alpha_{a,b}:=\begin{pmatrix}
a & b\\
-b & a
\end{pmatrix}: a^2+b^2\neq 0\right\},$$
contains matrices with arbitrary non-zero value of determinant. So, after possibly replacing $P$ by $\alpha_{a,b}P$ 
for an appropriate choice of $a,b \in \F{q}$, we can assume that $\det P=1$.
 
We will now show that the action of $\SL_2(\F{q})$ on $\SL_2(\F{q})/H$ does not have the weak EKR property. This will be done by finding a matrix
$$
B:=\begin{pmatrix}
a & b\\
c & d
\end{pmatrix}\in \SL_2(\F{q}),
$$
such that $\Tr(B)=\Tr(JB)=0$. 
Finding such a matrix involves solving the following equations:
\begin{equation}\label{EQPSL}
a+d =0, \quad b-c =0, \quad ad-bc =1.
\end{equation}
The system of Equations~\eqref{EQPSL} has a solution if and only if the equation $c^2+d^2=-1$ has a solution, which, by Lemma \ref{sum2}, is always the case. Let $c_{0}, d_{0}$ be such that $c_0^2+d_0^2=-1$. Clearly, $c_0,d_0\neq 0$, since $-1$ is not a square in $\f{q}$. Now, set
$$
B=\begin{pmatrix}
-d_0 & c_0\\
c_0 & d_0
\end{pmatrix}.
$$
Clearly, $B \neq \pm J, \pm I_2$, and therefore $S=\{\pm I_2,\pm J,B\}\subseteq \SL_2(\F{q})$ is an intersecting set since 
$$
S^{-1}S\subseteq \bigcup_{g\in \SL_2(\F{q})} H^g.
$$
Notice that $S$ has more than four elements. So $\SL_2(\F{q})$ does not have the weak EKR property in this case. \\

\noindent{\bf Case 2.} {\it $-1$ is a square in $\F{q}$, and $\gcd(2,q)=1$.}

Let $T$ be the subgroup of $\SL_2(\F{q})$ consisting of diagonal matrices, $J$ as in Case 1, and $S= T \cup \{ J \}$. It is easy to see that $J$ normalizes $T$ and the subgroup generated by $S$ is the union of $T$ and matrices of the form
\begin{equation}\label{t2}
J_{ \alpha}= \begin{pmatrix}
0  &  \alpha    \\
  -\alpha^{-1} & 0 \\
\end{pmatrix}, 
\end{equation}
for $ \alpha \in \F{q}^*$. The characteristic polynomial of $J_{ \alpha}$ is  
$x^2+ 1= ( x+\sqrt{-1})(x- \sqrt{-1})$, implying that $J_{ \alpha}$ is diagonalizable. We may assume (using a similar argument to in Case 1) that the determinant of the diagonalizing matrix is $1$. 
 This implies that 
\[  
S^{-1}S \subseteq \bigcup_{g\in \SL_2(\F{q})}T^g,
\]
which shows that the action of $\SL_2(\F{q})$ on $\SL_2(\F{q})/T$ does not have the weak EKR property. 

\noindent{\bf Case 3.} {\it $q>2$ is power of two.}

The argument in this case is similar to Case 1. Let $H$ be the $2$-element subgroup of $\SL_2(\F{q})$ generated by $J$ defined in case 1.
If $B \in \SL_2(\F{q}) \setminus \{ I_{2} \}$ satisfies $\Tr B=0$, then both the minimal and characteristic polynomial of $B$ will be given by $x^2+1=(x+1)^2$. This shows that there exists a matrix $ P \in \GL_2(\F{q})$ such that 
\[ P^{-1}BP= J=\begin{pmatrix}
0 & 1\\
1 & 0
\end{pmatrix}. \]
By the above remark, we can assume that $P \in \SL_{2}(\F{q})$. 
It remains to find a matrix $B$ so that $B\in\SL_2(\F{q})$, $B\neq J,I$ such that 
$\Tr(B)=\Tr(JB)=0$. Similar to case 1, to find such a matrix we need to solve the following equation with condition $c,d\neq 0$.
$$
c^2+d^2=1.
$$
Since $q$ is a power of $2$, the map $ x \mapsto x^2$ is surjective. Since $q>2$, by choosing $c \neq 0,1$, one can 
guarantee existence of $B$ distinct from both $I$ and $J$. Hence $S= \{ I, J, B \}$ is an intersecting set
for the action of $G$ on $G/H$. 
This finishes the proof for the projective special linear groups. The case of $\PSL_{3}(\f{3})$ has been studied in Remark \ref{remark-PSL3}.

We will now turn to the Suzuki groups. There are various ways to introduce the Suzuki groups. As we will need
to carry out explicit computations, a matrix representation of the group fits best our
purpose. The realization of the group given below follows G. Jones' paper~\cite{Jones}. 

Set $q=2^{2n+1}$, where $n \ge 1$ is an integer and let $ \theta: \f{q} \to \f{q} $ be the 
automorphism defined by $ \theta(x)=x^{2^{n+1}}$. Note that $ \theta(\theta(x))=x^{2}$, i.e., $\theta$ is
a square root of the Frobenius automorphism. For $a, \alpha, b , \beta \in \f{q} $ and $\gamma,c \in \f{q} ^{\ast}$ set,
\[ u( \alpha, a, \beta, b)= \begin{pmatrix}
1  & 0  & 0 & 0 \\
  \alpha & 1  & 0 & 0\\
\alpha a+ \beta & a & 1 &   0 \\
 \alpha^2a + \alpha \beta+ b & \beta & \alpha & 1
\end{pmatrix}, \qquad
h( \gamma, c)= \begin{pmatrix}
 \gamma c  & 0  & 0 & 0 \\
 0 &  \gamma  & 0 & 0\\
0 & 0 &  \gamma^{-1} &   0 \\
0 & 0 & 0 &  \gamma^{-1} c^{-1}
\end{pmatrix}. \]
Let $\tau$ denote the $4 \times 4$ matrix with $ \tau_{14}= \tau_{23}= \tau_{32}= \tau_{41}=1$, and all the other entries zero. For brevity, we will use the shorthands $ v( \alpha, \beta)= u( \alpha, \alpha^{\theta}, \beta, \beta^{\theta})$ and $ k( \gamma)= h( \gamma, \gamma^{\theta})$, for $ \gamma \neq 0$.
The Suzuki group is then defined by $\Suz_{q}= S \cup T$, where $S$ and $T$ are given by  
\[ S= \left\{v( \alpha, \beta)k( \gamma):
\alpha, \beta, \gamma \in \f{q} , \gamma \neq 0 \right\} \]
\[ T= \left\{ v( \alpha, \beta) \kappa( \gamma) 
\tau v( \delta, \epsilon):
\alpha, \beta, \gamma , \delta, \epsilon \in \f{q} , \gamma \neq 0 \right\}. \]
One can easily verify that
\begin{equation}\label{Suz1}
v( \alpha_{1}, \beta_{1}) v( \alpha_{2}, \beta_{2})= v( \alpha_{1}+ \alpha_{2}, \beta_{1}+ \beta_{2} + 
\alpha_{1}\alpha_{2}^{\theta}).
\end{equation}  
and 
\begin{equation} \label{suz-pro}
 k( \gamma)^{-1} v( \alpha, \beta) k( \gamma)= v( \alpha \gamma^{\theta}, \beta \gamma^{2} \gamma^{\theta}). 
\end{equation}
From \eqref{Suz1} and \eqref{suz-pro}, one can see that $S$ is a subgroup of $\Suz_{q}$ and the subsets $B$
and $H$ of $S$ defined by
\[ B= \left\{ v( \alpha, \beta):
\alpha, \beta \in \f{q}  \right\},  \quad
H= \left\{  k( \gamma): \gamma \in \f{q} , \gamma \neq 0 \right\}, \]
are both subgroups of $S$. Equation \eqref{suz-pro} also shows that $H$ normalizes $B$. 
Consider the subgroup $B_{0}\le B$ defined by
\[ B_{0}= \left\{ v( \alpha, \beta):
\alpha \in \f{2}, \beta \in \f{q} \right\}. \]
We claim that 
\[ B= \bigcup_{h \in H } B_{0}^{h}. \]
In order to prove this, we claim for a given pair $ ( \alpha, \beta) \in \f{q}  \times \f{q} $,  there exists 
$ \gamma \in \f{q}  \setminus \{ 0 \}$ such that $ \alpha \gamma^{\theta} \in \f{2}$. This is obvious for $ \alpha=0$.
For $ \alpha \neq 0$, the element $ \gamma \in \f{q}  \setminus \{ 0 \}$ must satisfy the equation 
$ \gamma^{\theta}= \alpha^{-1}$, which has the unique solution $\gamma= (\alpha^{\theta^{-1}})^{-1}.$
This claim shows that $B$ is an interesting set for the action of $G$ on $G/B_{0}$, and 
since $|B|>|B_{0}|$, by Lemma \ref{wekrl}, we see that $\Suz_{q}$ does not have the weak EKR property.

Finally, for the alternating groups, if $n=5$ the argument follows from the isomorphism $ A_5 \simeq \PGL_2(\f{4})$.
For $n\ge 6$, set $U= \{ \1, (12)(34) \}$ and $V= \{ \1, (12)(34), (12)(56), (34)(56) \}$, 
and use Lemma \ref{wekrl}. Therefore we have checked all cases in Lemma~\ref{EKR-simple}. This finishes the proof of Lemma~\ref{EKR-simple}. 
\end{proof}
We can now prove Theorem~\ref{Solvable}. 

\begin{proof}[Proof of Theorem~\ref{Solvable}]
Let $G$ be a minimal counter-example, that is, $G$ is a non-solvable group with the smallest size with the weak EKR property. 
Define $G^{0}=G$ and for $i \ge 0$, set $G^{i+1}=[G^{i}, G^{i}]$. Lemma~\ref{inducedEKR} 
shows that $G^{1}$ has the weak EKR property. If $G^{1}$ is solvable, then $G$ will be solvable too. Hence, by the minimality assumption, we have $G=[G,G]$, i.e., $G$ is perfect. Let $N$ a maximal proper normal subgroup of $G$. Then by Lemma~\ref{s-red}, $G/N$ has the weak EKR property. 
On the other hand, since $G$ is perfect, $G/N$ cannot be abelian. This shows that $G/N$ is a (non-abelian) finite simple
group. Once again, the minimality assumption shows that $N= \{ \1\}$ and hence $G$ is a non-abelian finite simple group. On the other hand, every proper subgroup of $G$ has the weak EKR property, hence 
it is solvable. This shows that $G$ is a minimal finite simple group with the weak EKR property which is contradictory to Lemma \ref{EKR-simple}.
\end{proof}
\subsection{Weak EKR for Nilpotent and Solvable Groups}
In this subsection we will prove Theorem~\ref{wekr}. We say that two subgroups $H$  and  $K$ of a group  $G$ are
complementary in $G$, if $H\cap K=\{\1\}$ and $KH=G$. A finite group  $G$ is said  to  be {\it complemented} if 
for every subgroup  $H \le G$ there exists a subgroup $K \le G$ such that $H$ and $K$ are complementary in $G$. 
\begin{theorem}[Hall~\cite{hall}]\label{Hall}
A  group of finite order is complemented if  and only if it is
isomorphic with a subgroup of a direct product of groups of square-free order.
\end{theorem}
 
\begin{proof}[Proof of Theorem \ref{wekr}]
Theorem~\ref{Hall} and Corollary~\ref{comp} establish the second part of the claim.

For the first part, let us assume that $G$ is nilpotent and $H$ is a subgroup of $G$. We will prove by induction on $|G|$ that the action of $G$ on $G/H$ has the weak EKR property. 
Let $A= H \cap Z(G)$. Note that $A$ is a normal subgroup of $G$ which is included in $H$. 
If $|A|>1$, then by induction hypothesis the action of $G/A$ on $G/H$ has the weak EKR property, 
and part (3) of Lemma \ref{inducedEKR} implies that the action of $G$ on $G/H$ also has the weak EKR
property. We can hence assume that $H \cap Z(G)= \{ \1 \}$. Let $\pi: G \to G/Z(G)$ be the quotient map.
By the induction hypothesis, the action of $\pi(G)$ on $\pi(G)/\pi(H)$ has the weak EKR property. 
Let $S \subseteq G$ be an intersecting set for the action of $G$ on $G/H$. 
Since $\pi(S)^{-1} \pi(S) \subseteq \bigcup_{g \in \pi(G)} \pi(H)^{g}$, the image
$\pi(S)$ is an intersecting set for the action of $\pi(G)$ on $\pi(G)/\pi(H)$. Moreover,
$S^{-1}S \cap Z(G) \subseteq \bigcup_{g \in G} (H\cap Z(G))^g = \{ \1 \}$.
Hence, the restriction of $\pi$ to 
$S$ is injective and $|\pi(S)|=|S|$. Now, from the assumption that the action of $\pi(G)$ on $\pi(G)/\pi(H)$ 
has the weak EKR property, it follows that $|\pi(S)| \le |\pi(H)|=[H:H \cap Z(G)]=|H|$. This finishes the proof. 
\end{proof}
\section{Nilpotent Groups with the strong EKR Property}\label{ST}
In this section, we will give a proof for Theorem \ref{classification_S}.  
The following lemma is standard and follows from an inductive argument:

\begin{lemma}\label{two-step}
Let $G$ be a two-step nilpotent group and $x, y \in G$ with $ [x,y]=z$. For
$m, n \in \mathbb{Z}$, we have
\[ [x^m, y^n]=z^{mn}. \]
\end{lemma}
We will also need the following lemma about modification of large cosets.
The lemma states the unsurprising fact that the cosets of subgroup of large index are
rigid, in the sense that a small modification in a coset never yields a coset of a conjugate of the same subgroup. 
\begin{lemma}\label{mod-int}
Let $x$ be an element of prime order $p>2$ in a finite group $G$, and 
$H$ the subgroup generated by $x$. For $t \in G \setminus H$, the set $S= (H \setminus \{\1  \})
\cup \{ t \}$ is not a coset of a conjugate of $H$. 
\end{lemma}
\begin{proof}
If this is not the case, there exists $g, w \in G$ such that 
$S= H^gw$. Note that for each $ 1 \le i \neq j \le p-1$, we have
\[ x^{i-j}= x^i x^{-j} \in (H^gw)(w^{-1}H^g) \subseteq H^g.  \]
Since $p>2$, the difference $i-j$ takes all the non-zero residues module $p$. 
Hence 
\[ |H^g \cap H| \ge p-1. \]
This shows that $H=H^g$, hence $S= Hw$. Let $h_0\in H\cap S$. Since $S=Hw$, we have $h_0=h_1w$, for some $h_1\in H$. Therefore $w\in H$, which implies that $S=H$, which is a contradiction since $\1\not\in S$.   
\end{proof}
We can now state and prove the main theorem. 
\begin{proof}[ Proof of Theorem~\ref{classification_S}]
Let $G$ be a minimal counter-example to Theorem \ref{classification_S}. 
Since $G$ is nilpotent we can write $G= \prod_p S_p$, where $S_p$ is the unique 
$p$-sylow subgroup of $G$. By minimality of $G$ and Lemma \ref{s-red}, we have $G=S_{p}$ for some $p>2$, i.e., $G$ is a non-abelian $p$-group for some $p>2$. Using Lemma \ref{s-red}, we will make a series of further reductions. 
As $|G/Z(G)|<|G|$, the quotient $G/Z(G)$ is abelian, or equivalently, $[G,G] \subseteq Z(G)$. We claim that $[G,G]$ has exactly $p$ elements. Assume by way of contradiction that $|[G,G]|>p$ and let $A\unlhd [G,G]
\subseteq Z(G)$ be 
a non-trivial proper subgroup. Then $G/A$ is a non-abelian $p$-group with the strong EKR property which contradicts the minimality of $G$. Our next claim is that $G\setminus Z(G)$ contains an element of order $p$. To prove this claim, we use a theorem (see~\cite{Rob} Theorem 5.3.6), stating that when $p$ is an odd prime, then any non-abelian $p$-group has at least two subgroups of order $p$. Now, if all elements of order $p$ are included in the center, then we can choose a subgroup $L$ of order $p$ with $[G,G] \nsubseteq L\subseteq Z(G)$.
Then $G/L$ is a non-abelian $p$-group with the strong EKR property, which by minimality of $G$ is again a contradiction. Choose an element $x$ of order $p$ in $G \setminus Z(G)$ and for each $1 \le j \le p-1$, set
\[ C_j= \{ [x^j,y]: y \in G \}. \]
Since $x \not \in Z(G)$, we have $ C_1 \neq \{ \1 \}$. Choose $t \in C_1 \setminus 
\{ \1 \}$ and let $t=[x,y]$. Clearly, $t$ is a central element of order $p$. Since $[G,G] \subseteq Z(G)$, by Lemma \ref{two-step}, for each $ 1\le i \le p-1$, we have
\[ C_i \supset \{ [x^i, y^j]: 0 \le j \le p-1 \} = \{t^{ij}: 0 \le j \le p-1 \}
= [G,G]. \]
This shows that for each $ 1 \le i \le p-1$, there exists $y_i$ such that 
$[x^i, y_i]=t$, or equivalently,
\[ y_ix^{-i}y_i^{-1}=tx^{-i}. \]
This proves that the set $\{ x^j: 1 \le j \le p-1 \} \cup \{ t^{-1} \}$ 
is an intersecting set for the action of $G$ on $G/H$, where $H$ is the subgroup generated by $x$, and by Lemma \ref{mod-int} is not a coset of a conjugate
of the subgroup generated by $x$. This contradiction proves the theorem.
\end{proof}
\section{Group Actions with the Weak and Strong EKR}\label{ACT}
Let $q$ be a power of a prime $p$. In this section we consider sevaral actions of 
the groups $\PGL_n(\F{q})$ and $\PSL_{n}(\F{q})$ and prove the weak and strong EKR property for them.

\subsection{Action on the Projective Space}
In this subsection, we will consider the standard actions of $\PGL_n(\F{q})$ and $\PSL_{n}(\F{q})$ on the projective space $\P{n-1}$ and give a proof of Theorem \ref{mainPSL}.

First, we will briefly recall some elementary facts that will be needed for the proof.  
Recall that $\F{q^n}$ is a $\F{q}$-vector space of dimension $n$. Every non-zero $x\in\F{q^n}$ defines 
a non-singular $\F{q}$-linear transformation $\varphi_x: \F{q^n} \rightarrow \F{q^n}$ via $\varphi_x(a) = ax$
and the map $x \mapsto \varphi_{x}$ defines a group homomorphism $
\Phi: \F{q^n}^* \rightarrow  \PGL_n(\F{q})$ by $\Phi(x) = \varphi_x Z$, where $Z$ denotes the centre of $\GL_n(\F{q})$ consisting of scaler matrices. The kernel of this map is easily seen to be $\F{q}^*$.
We will use the following lemma which can be found in \cite{neukirch} (see Definition 2.5 and Proposition 2.6). 

\begin{lemma}\label{}
Let $E/F$ be a Galois extension of degree $n$ with the Galois group $G$. Then for any $x\in E$ we have
$$
\det(tI-\varphi_x)=\prod_{\sigma\in G}(t -\sigma(x)).
$$ 
Therefore $\{\sigma(x): \sigma\in G\}$ is the multi-set of eigenvalues of the multiplication map $\varphi_x$.
\end{lemma}

The Galois group of $\F{q^n}$ over $\F{q}$ is generated by the Frobenius automorphism $\sigma_q$
defined by $\sigma_{q}(x):= x^{q}$. This shows that the set of all eigenvalues of $\varphi_x$ for $ x \in \F{q^n}^*$ is given by 
$$
\left\{\sigma_q^\ell(x): 0\leq \ell\leq {n-1}\right\}=\left\{x^{q^\ell}: 0\leq \ell\leq {n-1} \right\}.
$$ 
 We are now ready to give a proof of Theorem~\ref{mainPSL}.
\begin{proof}[Proof of Theorem~\ref{mainPSL}]
Let $\zeta\in\F{q^n}^*$ be a generator of the cyclic group $\F{q^n}^*$ and set $m:=q^{n-1}+\dots+q+1$. We will show that the set
$$
T:=\left\{\Phi(\zeta^j): 0\leq j\leq m-1 \right\},
$$
is an independent set in the graph $\Gamma_{\PGL_n(\F{q}),\P{n-1}}$ defined in Section~\ref{BASE}. A projective transformation 
has a fixed point on the projective space if and only if the corresponding linear transformation has an eigenvalue in $\F{q}$. So, it
suffices to show that for $0\leq i\neq j\leq m-1$, none of the eigenvalues of $\varphi_{\zeta^{j-i}}$ lies in $\F{q}$. Suppose by way of contradiction that $\zeta^{(j-i)q^{\ell}}\in \F{q}$ for some $0\leq i\neq j\leq m-1$, and $0\leq \ell\leq {n-1}$. 
This implies that 
$
\zeta^{(j-i)q^\ell(q-1)}=1$ and hence $q^n-1\mid (j-i)q^\ell(q-1)$. This easily reduces to $m|j-i$,
which is a contradiction, since $1\leq |j-i|\leq m-1$. Moreover, for any $1\leq j\leq m-1$, 
we have $\zeta^j\not\in \F{q}$, and therefore $|T|=m=|\P{n-1}|$. 
Now, Lemma~\ref{cliq}, shows that an intersecting set in $\PGL_n(\F{q})$ has the size at most 
$[\PGL_n(\F{q}): T]=|\PGL_n(\F{q})|/|\P{n-1}|$.

Let us now consider the group $\PSL_n(\F{q})$. First, recall that since $\gcd(n,q-1)=1$, the natural surjection $\SL_n(\F{q}) \to \PSL_n(\F{q})$
is indeed an isomorphism. It is easy to see that the order of $\zeta^{q-1}$ is
equal to $m$. Define
$$
\mu_{m}:=\{\zeta^{(q-1)j}: 0\leq j\leq m-1\}.
$$
Since the product of the eigenvalues of $\varphi_{\zeta^{(q-1)j}}$ is $\zeta^{j(q-1)m}=1$,
we obtain a group homomorphism $\Phi_1: \mu_{m} \rightarrow \PSL_n(\F{q}).$
Set $T_{m}:=\Phi_1(\mu_{m})$. We claim that $T_m$ is an independent set in the graph $\Gamma_{\PSL_n(\F{q}),\P{n-1}}$. To show this, note that the eigenvalues of $\varphi_{\zeta^{(q-1)j}}$, $0\leq j\leq m-1$, are given by 
$$\lambda_\ell:=\zeta^{j(q-1)q^\ell}, \qquad 0\leq \ell\leq n-1.$$

If $  \lambda_\ell \in \F{q}$, then $ \lambda_\ell^{q-1}=1$, which implies that $q^n-1\mid j(q-1)^2q^\ell$.
Since $\gcd(n,q-1)=1$, and $m \equiv n \pmod{q-1}$, we have $\gcd(m, q-1)=1$. This together with $\gcd(q^n-1,q^\ell)=1$ gives $m|j$, which is a contradiction. This shows that $T_m$ is an independent set of size $|T_m|=m$
in the graph $\Gamma_{\PSL_n(\F{q}),\P{n-1}}$. Now Lemma~\ref{cliq} establishes the result. 
\end{proof}  
\subsection{Action on $G/U$}
In this subsection we will prove Theorem~\ref{Unipotent}. Recall that a matrix $g \in \GL_{2}(\f{q})$ is 
unipotent if and only if $(g-I_2)^{2}=0$. We will denote the set of unipotent elements by ${\mathcal U}$.
For simplicity, we will refer to conjugates of $U$ as unipotent subgroups. It is clear that any element of a unipotent subgroup is a unipotent element. One can also show that any unipotent element is contained in a unipotent subgroup, but we do not need this fact here. The 
following lemma is the key to the proof of Theorem~\ref{Unipotent}.

\begin{lemma}\label{uni}
Let $U_{1}$ and $U_{2}$ be distinct unipotent subgroups of $G=\GL_2(\f{q})$. Then
\begin{enumerate}
\item $U_{1}\cap U_{2}= \{ I_{2} \}$.
\item If $g_{i} \in U_{i}$, 
$i=1,2$ be such that $g_{1}g_{2}$ is a unipotent matrix. Then $g_{1}=I_{2}$ or $g_{2}=I_{2}$.
\end{enumerate}
\end{lemma}

\begin{proof}
For part (1), we first claim that if $V$ is a unipotent subgroup and $g \in V \setminus \{ I_2 \}$, then 
\[ V= C_{G}(g) \cap {\mathcal U}. \]
In other words, $V$ consists of those unipotent elements of $G$ that commute with $g$. It suffices to prove
this statement for $V=U$. For $x \neq 0$, expanding the equation
\[ \begin{pmatrix}
1  & x    \\
 0 & 1 \\
\end{pmatrix} \begin{pmatrix}
a  & b    \\
 c & d \\
\end{pmatrix} = \begin{pmatrix}
a  & b    \\
 c & d \\
\end{pmatrix} \begin{pmatrix}
1  & x    \\
 0 & 1 \\
\end{pmatrix} \]
results in $c=0$ and $a=d$. Hence if $hg=gh$, then 
\[ h= \begin{pmatrix}
a  & b    \\
 0 & a \\
\end{pmatrix}. \]
If $h$ is unipotent, then $a=1$, and hence $h \in U$. This proves that $C_{G}(g) \cap {\mathcal U} \subseteq U$.
The reverse inclusion is obvious. (1) is immediate from the this claim.

For part (2), without loss of generality, we can assume that $U_{1}=U$, hence $g_{1}$ is an upper-triangular matrix, which 
we assume is not the identity matrix. 
Then we can write
\[g_1g_2=\begin{pmatrix}
1  & a    \\
 0 & 1 \\
\end{pmatrix} \begin{pmatrix}
x & y    \\
z & 2-x \\
\end{pmatrix} = \begin{pmatrix}
x+az  & y+a(2-x)    \\
 z & 2-x \\
\end{pmatrix} \]
and hence from $\Tr(g_{1}g_{2})=2+az=2$, and $g_{1} \neq I_2$ we deduce that $z=0$. This together
with the fact that $\det g_{2}=x(2-x)=1$ implies that $x=1$, hence $g_{2}=I_2$.

\end{proof}

\begin{proof}[Proof of Theorem~\ref{Unipotent}]
We will show that if  $S \subseteq G$ is an intersecting subset containing $I_{2}$, then $S$ is contained in a unipotent subgroup. Take distinct $g_1, g_2 \in S \setminus \{ I_{2} \}$. Since $S$ is intersecting containing the identity element, all
elements $g_1^{-1}, g_1^{-1}g_2$ and $g_2$ are contained in unipotent subgroups, and hence Lemma \ref{uni} shows that $g_1$ and $g_2$ are in the same unipotent subgroup. This shows that $S \subseteq U$ for a unipotent subgroup $U$, proving the result. 
\end{proof}
\section{Acklowledgement}
Part of this research was carried out during the second author's visit at the University of Ottawa. The second author wishes to thank Vadim Kaimanovich for his warm hospitality. The first author was supported by postdoctoral fellowship from the University of Ottawa during the completion
of this work.  The first author wishes to thank his supervisor Vadim Kaimanovich for the financial support and the pleasant working environment. Authors would like to especially thanks the referees for several detailed comments that lead 
to improving the exposition of the paper and correcting some inaccuracies.

\bibliographystyle{plain}

\end{document}